\documentclass{amsart}
\usepackage{amsfonts}
\usepackage{amsmath}
\usepackage{amssymb}
\usepackage{hyperref}

\newtheorem{theorem}{Theorem}[section]
\theoremstyle{plain}

\newtheorem{Application}{Application}[section]

\newtheorem{corollary}{Corollary}[section]

\newtheorem{definition}{Definition}[section]

\newtheorem{lemma}{Lemma}[section]

\newtheorem{remark}{Remark}[section]

\numberwithin{equation}{section}
\input{tcilatex}

\begin{document}
\title[Bivariate log-convexities of the more extended means]{Bivariate
log-convexity of the more extended means and its applications}
\author{Zhen-Hang Yang}
\address{System Division, Zhejiang Province Electric Power Test and Research
Institute, Hangzhou, Zhejiang, China, 310014}
\email{yzhkm@163.com}
\date{March 28, 2010}
\subjclass[2000]{Primary 26B25, 26D07; Secondary 26E60,26A48}
\keywords{Stolarsky means, Gini means, More\ extended means, Two-parameter
homogeneous function, bivariate log-convexity, inequality}
\thanks{This paper is in final form and no version of it will be submitted
for publication elsewhere.}

\begin{abstract}
In this paper, the bivariate log-convexity of the two-parameter homogeneous
function in parameter pair is vestigated. From this the bivariate
log-convexity of the more extended means with respect to a parameter pair is
solved. It follows that Stolarsky means, Gini means, two-parameter identric
(exponential) means and two-parameter Heronian means are all bivariate
log-concave on $\mathbb{[}0\mathbb{,\infty )}^{2}$ and log-convex on $%
\mathbb{(-\infty ,}0\mathbb{]}^{2}$ with respect to parameters. Lastly, some
classical and new inequalities for means are given.
\end{abstract}

\maketitle

\section{\textsc{Introduction}}

\ Throughout the paper we denote by $\mathbb{R}_{+}=(0,\infty )$, $\mathbb{R}%
_{-}=(-\infty ,0)$ and $\mathbb{R=(-\infty ,\infty )}$.

For $a,b\in \mathbb{R}_{+}$ with $a\neq b$ the Stolarsky means $S_{p,q}(a,b) 
$ were defined by Stolarsky \cite{Stolarsky.48.1975} as

\begin{equation}
S_{p,q}(a,b)=\left\{ 
\begin{array}{ll}
\left( {\dfrac{q(a^{p}-b^{p})}{p(a^{q}-b^{q})}}\right) ^{\frac{1}{p-q}} & 
\text{if }p\neq q,pq\neq 0, \\ 
\left( {\dfrac{a^{p}-b^{p}}{p(\ln a-\ln b)}}\right) ^{\frac{1}{p}} & \text{%
if }p\neq 0,q=0, \\ 
\left( {\dfrac{a^{q}-b^{q}}{q(\ln a-\ln b)}}\right) ^{\frac{1}{q}} & \text{%
if }p=0,q\neq 0, \\ 
\exp \left( \dfrac{a^{p}\ln a-b^{p}\ln b}{a^{p}-b^{p}}-\dfrac{1}{p}\right) & 
\text{if }{p=q\neq 0,} \\ 
\sqrt{ab} & \text{if }{p=q=0.}%
\end{array}%
.\right.  \label{1.1}
\end{equation}%
Also, $S_{p,q}(a,a)=a$. The Stolarsky means contain many famous means, for
example, $S_{1,0}(a,b)=L(a,b)$ --the logarithmic mean, $S_{1,1}(a,b)=I(a,b)$
--the identric (exponential) mean, $S_{2,1}(a,b)=A(a,b)$ --arithmetic mean, $%
S_{3/2,1/2}(a,b)=He(a,b)$ --Heronian mean, $%
S_{2p,p}(a,b)=A^{1/p}(a^{p},b^{p})=A_{p}$ --the $p$-order power mean, $%
S_{p,0}(a,b)=L^{1/p}(a^{p},b^{p})=L_{p}$ --the $p$-order logarithmic mean, $%
S_{p,p}(a,b)=I^{1/p}(a^{p},b^{p})=I_{p}$ --the $p$-order identric
(exponential) mean, etc.

Stolarsky means are also called "extended means" \cite{Leach.85.1978}, or
"difference means" \cite{Pales.131.1988b}, and belong to the "two-parameter
family of bivariate means".

Another well-known two-parameter family of bivariate means was introduced by
C. Gini in \cite{Gini.13.1938}. That is defined as

\begin{equation}
G_{p,q}(a,b)=\left\{ 
\begin{array}{ll}
\left( {\dfrac{a^{p}+b^{p}}{a^{q}+b^{q}}}\right) ^{1/(p-q)} & \text{if }%
p\neq q, \\ 
\exp \left( \dfrac{a^{p}\ln a+b^{p}\ln b}{a^{p}+b^{p}}\right) & \text{if }{%
p=q.}%
\end{array}%
\right.  \label{1.2}
\end{equation}%
The Gini means also contain many famous means, for example, $%
G_{1,0}(a,b)=A(a,b)$ --arithmetic mean, $G_{1,1}(a,b)=Z(a,b)$ --the
power-exponential mean, $G_{p,0}(a,b)=A^{1/p}(a^{p},b^{p})=A_{p}$ --the $p$%
-order power mean, $G_{p,p}(a,b)=Z^{1/p}(a^{p},b^{p})=Z_{p}$ --the $p$-order
power-exponential mean, etc.

It is clear that Stolarsky and Gini means both have the form of $\left( 
\frac{f(a^{p},b^{p})}{f(a^{q},b^{q})}\right) ^{1(p-q)}$, where $f$ is a
positive homogeneous function and $p\neq q$. A function in this form is
called a "two-parameter homogeneous function generated by $f$" by Yang \cite%
{Yang.6(4).101.2005} and denote by $\mathcal{H}_{f}(p,q;a,b)$. Furthermore,
if it is a mean of positive numbers $a$ and $b$ for every $p,q$, then all
these means are members of "two-parameter family of bivariate means" and,
for short, "two-parameter $f$-means".

Substituting $f=S_{r,s}(x,y)$ into $\mathcal{H}_{f}(p,q;a,b)$, Yang \cite%
{Yang.8(2005).3.Art.8, Yang.149286.2008} defined a class of more general
two-parameter means without trying to prove it, that is, "four-parameter
homogeneous means" denote by $\boldsymbol{F}$$(p,q;r,s;a,b)$. Witkowski
later proved that $\boldsymbol{F}$$(p,q;r,s;a,b)$ are means of positive
numbers $a$ and $b$ in \cite[6.4]{Witokwski.12(1).2009.Art.3}. For sake of
statement, we recall the four-parameter means $\boldsymbol{F}$$(p,q;r,s;a,b)$
as follows.

\begin{definition}[\protect\cite{Yang.149286.2008, Witokwski.12(1).2009}]
Let $(p,q),(r,s)\in \mathbb{R}^{2},(a,b)\in \mathbb{R}_{+}$. Then $%
\boldsymbol{F}$$(p,q;r,s;a,b)$ are called four-parameter means if for $a\neq
b$%
\begin{equation}
\text{$\boldsymbol{F}$}(p,q;r,s;a,b)=\left( {\frac{L(a^{pr},b^{pr})}{%
L(a^{ps},b^{ps})}\frac{L(a^{qs},b^{qs})}{L(a^{qr},b^{qr})}}\right) ^{\frac{1%
}{(p-q)(r-s)}}\text{ if }pqrs(p-q)(r-s)\neq 0,  \label{1.3}
\end{equation}%
or%
\begin{equation}
\boldsymbol{F}(p,q;r,s;a,b)=\left( \frac{a^{pr}-b^{pr}}{a^{ps}-b^{ps}}{\frac{%
a^{qs}-b^{qs}}{a^{qr}-b^{qr}}}\right) ^{\frac{1}{(p-q)(r-s)}}\text{ if }%
pqrs(p-q)(r-s)\neq 0;  \label{1.4}
\end{equation}%
$\boldsymbol{F}(p,q;r,s;a,b)$ are defined as their corresponding limits if $%
pqrs(p-q)(r-s)=0$, for example: 
\begin{eqnarray*}
\boldsymbol{F}(p,p;r,s;a,b) &=&\underset{q\rightarrow p}{\lim }\boldsymbol{F}%
(p,q;r,s;a,b)=\left( {\frac{I(a^{pr},b^{pr})}{I(a^{ps},b^{ps})}}\right) ^{%
\frac{1}{p(r-s)}}\text{ if }prs(r-s)\neq 0,p=q, \\
\text{$\boldsymbol{F}$}(p,0;r,s;a,b) &=&\underset{q\rightarrow 0}{\lim }%
\boldsymbol{F}(p,q;r,s;a,b)=\left( {\frac{L(a^{pr},b^{pr})}{L(a^{ps},b^{ps})}%
}\right) ^{\frac{1}{p(r-s)}}\text{ if }prs(r-s)\neq 0,q=0, \\
\text{$\boldsymbol{F}$}(0,0;r,s;a,b) &=&\underset{p\rightarrow 0}{\lim }%
\text{$\boldsymbol{F}$}(p,0;r,s;a,b)=G(a,b)\text{ if }rs(r-s)\neq 0,p=q=0,
\end{eqnarray*}%
where $L(x,y),I(x,y)$ denote logarithmic mean and identric (exponential)
mean, respectively, $G(a,b)=\sqrt{ab}$. And, $\boldsymbol{F}(p,q;r,s;a,a)=a$.
\end{definition}

As has been mentioned above, the four-parameter homogeneous means $%
\boldsymbol{F}$$(p,q;r,s;a,b)$ are generated by Stolarsky means or extended
means, and clearly contain many two-parameter means, for instance (see \cite[%
Table 1]{Yang.149286.2008}),

\noindent $\boldsymbol{F}(p,q;1,0;a,b)=S_{p,q}(a,b)$ --two-parameter
logarithmic means, that is, Stolarsky means;

\noindent $\boldsymbol{F}(p,q;2,1;a,b)=G_{p,q}(a,b)$ --two-parameter
arithmetic means, that is, Gini means;

\noindent $\boldsymbol{F}(p,q;1,1;a,b)=I_{p,q}(a,b)$ --two-parameter
identric (exponential) means, where $I_{p,q}(a,b)$ is defined as%
\begin{equation}
I_{p,q}(a,b)=\left\{ 
\begin{array}{ll}
\left( {\dfrac{I(a^{p},b^{p})}{I(a^{q},b^{q})}}\right) ^{1/(p-q)} & \text{if 
}p\neq q, \\ 
Y^{1/p}(a^{p},b^{p}) & \text{if }{p=q,}%
\end{array}%
\right.  \label{1.5}
\end{equation}%
here $Y(a,b)=Ie^{{1-G^{2}/L^{2}}},I,L,G$ are identric (exponential),
logarithmic, geometric mean, respectively;

\noindent $\boldsymbol{F}(p,q;3/2,1/2;a,b)=He_{p,q}(a,b)$ --two-parameter
Heronian means, where $He_{p,q}(a,b)$ is defined as%
\begin{equation}
He_{p,q}(a,b)=\left\{ 
\begin{array}{ll}
\left( {\dfrac{a^{p}+(\sqrt{ab})^{p}+b^{p}}{a^{q}+(\sqrt{ab})^{q}+b^{q}}}%
\right) ^{1/(p-q)} & \text{if }p\neq q, \\ 
a^{\frac{a^{p}+(\sqrt{ab})^{p/2}}{a^{p}+(\sqrt{ab})^{p}+b^{p}}}b^{\frac{(%
\sqrt{ab})^{p/2}+b^{p}}{a^{p}+(\sqrt{ab})^{p}+b^{p}}} & \text{if }{p=q,}%
\end{array}%
\right.  \label{1.6}
\end{equation}

\noindent etc. Hence, we prefer to call the four-parameter homogeneous means 
$\boldsymbol{F}$$(p,q;r,s;a,b)$ as "more extended means" in what follows.

Study of two-parameter means containing Stolarsky and Gini means is
interesting and has attracted the attention of a considerable number of
mathematicians.

The monotonicity of $S_{p,q}(a,b)$ with respect to parameter $p$ and $q$ was
verified by Leach and Sholander \cite{Leach.85.1978}. Qi (\cite%
{Guo.29(2).1999} \cite{Qi.126(11).1998} \cite{Qi.211.1997} \cite%
{Qi.22(2).1999} \cite{Qi.9(1).1999} \cite{Qi.224(2).1998}) \cite%
{Qi.3(3).2000}) has also researched by using different ideas and simpler
methods. Gini \cite{Gini.13.1938} showed $G_{p,q}(a,b)$ increases with
either $p$ or $q$ (see also \cite{Farnsworth.93(8).1986}). Yang \cite%
{Yang.6(4).101.2005} gave two unified discriminants for monotonicities of
the two-parameter homogeneous functions $\mathcal{H}_{f}(p,q;a,b)$ and
obtained the same monotonicity results again. Moreover, Yang \cite%
{Yang.149286.2008} has proved that the more extended means $\boldsymbol{F}%
(p,q;r,s;a,b)$ are strictly increasing (decreasing) in either $p$ or $q$ if $%
r+s>(<)0$.

The comparison problem for Stolarsky means $S_{p,q}(a,b)\leq S_{r,s}(a,b)$ ($%
a,b\in \mathbb{R}_{+}$) was solved by Leach and Sholander \cite%
{Leach.92.1983}. P\'{a}les \cite{Pales.131.1988b} presented a new proof for
this result. In \cite{Pales.32.1989} P\'{a}les solved the same comparison
problem on any subinterval $(\alpha ,\beta )$ of $\mathbb{R}_{+}$. The same
comparison problem of Gini means $G_{p,q}(a,b)$ was dealt with by P\'{a}les 
\cite{Pales.131.1988a, Pales.103.1992}. Applying unified comparison theorem
given in \cite{Pales.32.1989}, a more generalized comparison problem for
more extended means $\boldsymbol{F}$$(p,q;r,s;a,b)\leq \boldsymbol{F}$$%
(u,v;r,s;a,b)$ ($a,b\in \mathbb{R}_{+}$) has been solved by Witkowski
recently (see \cite{Witokwski.12(1).2009}). Other references involving these
comparison problems can be found in \cite{Dodd.12.1971}, \cite%
{Carlson.79.1972}, \cite{Brenner.3.1978}, \cite{Burk.92.1985}, \cite%
{Pales.32.1989}, \cite{Losonczi.7(1).2002}, \cite{Neuman.15(2).2003}, \cite%
{Czinder.9(4).2006}

Losonczi and P\'{a}les \cite{Losonczi.62.1996, Losonczi.126.1998}
established the Minkowski inequalities for the Stolarsky in 1996 and Gini
means in 1998, respectively. Since symmetry of the two classes of means, and
so they perfectly dealt with the bivariate convexities with respect to $%
(a,b) $ actually. The H\"{o}lder type inequality, that is, geometrical
convexity for Stolarsky means $S_{p,q}(a,b)$ was proved by Yang \cite%
{Yang.6(4).17.2005, Yang.4(34).2010}.

For the univariate log-convexity of Stolarsky means with respect to
parameters, Qi first proved that

\textbf{Theorem Q (}\cite[Theorem 1]{Qi.130(6)(2002)}\textbf{). }\textit{If }%
$p,q\geq 0$\textit{, then }$S_{p,q}(a,b)$\textit{\ are log-concave in both }$%
p$\textit{\ and }$q$\textit{. If }$p,q\leq 0$\textit{, then }$S_{p,q}(a,b)$%
\textit{\ are log-convex in both }$p$\textit{\ and }$q$\textit{.}

In 2003, Neuman and S\'{a}ndor \cite{Neuman.14(1).2003} showed this property
is also true for Gini means $G_{p,q}(a,b)$. Four years later Yang \cite%
{Yang.10(3).2007} presented an unified treatment for log-convexity of the
two-parameter homogeneous functions $\mathcal{H}_{f}(p,q;a,b)$ in either $p$
or $q$, and obtained the Qi's and Neuman and S\'{a}ndor's results again. In
2008, this unified treatment was applied to the more extended means $%
\boldsymbol{F}$$(p,q;r,s;a,b)$ and obtained the following

\textbf{Theorem Y (}\cite[Theorem 2.3]{Yang.149286.2008}\textbf{). }\textit{%
If }$r+s>(<)0$\textit{, then }$F(p,q;r,s;a,b)$\textit{\ are strictly
log-concave (log-convex) in either }$p$\textit{\ or }$q$\textit{\ on }$%
(0,\infty )$\textit{\ and log-convex (log-concave) on }$(-\infty ,0)$\textit{%
.}

In 2010 Yang \cite{Yang.1(1).2010} further proved that\ for fixed $q\in 
\mathbb{R}$ the two-parameter symmetric homogeneous function $\mathcal{H}%
_{f}(p,q;a,b)$ is strictly log-convex (log-concave) in $p$ on $(\frac{|q|-q}{%
2},\infty )$ and log-concave (log-convex) on $(-\infty ,-\frac{|q|+q}{2})$
if ${{\mathcal{J}=(x-y)(x}\mathcal{I}{)_{x}<(>)0}}$, where{\ }${\mathcal{I}%
=(\ln f)_{xy}}$. This improved the Qi's and Neuman and S\'{a}ndor's results.

It might be because of complexity, up to now the research for log-convexity
of the two-parameter means in parameters has only been limited to the
univariate cases. In other words, that research for bivariate
log-convexities of the two-parameter means in parameter pair remains a blank.

The aim of this paper is to investigate the bivariate log-convexity of the
more extended means $\boldsymbol{F}$$(p,q;r,s;a,b)$ in a parameter pair $%
(p,q)$. Our approach is to use the well properties of homogeneous functions
and an integral representation for the two-parameter homogeneous functions%
\emph{\ }$\mathcal{H}_{f}(p,q;a,b)$ to obtain an unified discriminants for
the bivariate log-convexity in parameter pair $(p,q)$. From this the problem
for bivariate log-convexity of the more extended means $\boldsymbol{F}$$%
(p,q;r,s;a,b)$ with respect to a parameter pair is successfully solved. As
special cases, the bivariate log-convexity of the Stolarsky means, Gini
means, two-parameter identric (exponential) means and two-parameter Heronian
means with respect to parameters easily follows.

Additionally, we also investigate a special two-parameter homogeneous
function but not a mean passingly, whose bivariate log-convexity can
generate some interesting and useful inverse inequalities for Stolarsky and
Gini means.

\section{\textsc{Bivariate Log-convexity of Two-parameter Homogeneous
Functions }}

In this section, we will deal with bivariate log-convexity of two-parameter
homogeneous function generated by $f$ with respect to parameters. For this
end, we first recall definition of two-parameter homogeneous function
generated by $f$ as follows.

\begin{definition}
\label{d2.1}Let $f$: $\mathbb{R}_{+}^{2}\backslash \{(x,x),x\in \mathbb{R}%
_{+}\}\rightarrow \mathbb{R}_{+}$ be a homogeneous, continuous function and
has first partial derivatives. Then the function $\mathcal{H}_{f}:\mathbb{R}%
^{2}\times \mathbb{R}_{+}^{2}\rightarrow \mathbb{R}_{+}$ is called a
homogeneous function generated by $f$ with parameters $p$ and $q$ if $%
\mathcal{H}_{f}$ is defined by for $a\neq b$ 
\begin{eqnarray}
\mathcal{H}_{f}(p,q;a,b) &=&\left( {\frac{f(a^{p},b^{p})}{f(a^{q},b^{q})}}%
\right) ^{1(p-q)}\text{ if }pq(p-q)\neq 0,  \label{2.1} \\
\mathcal{H}_{f}(p,p;a,b) &=&\exp \left( \frac{a^{p}f_{x}(a^{p},b^{p})\ln
a+b^{p}f_{y}(a^{p},b^{p})\ln b}{f(a^{p},b^{p})}\right) \text{ if }p=q\neq 0,
\label{2.2}
\end{eqnarray}%
where $f_{x}(x,y)\ $and $f_{y}(x,y)\ $denote first-order partial derivatives
with respect to first and second component of $f(x,y)$,$\ $respectively; if $%
\underset{y\rightarrow x}{\lim }f(x,y)$ exits and is positive for all $x\in 
\mathbb{R}_{+}$, then further define%
\begin{eqnarray}
\mathcal{H}_{f}(p,0;a,b) &=&\left( {\frac{f(a^{p},b^{p})}{f(1,1)}}\right)
^{1/p}\text{ if }p\neq 0,q=0,  \label{2.3} \\
\mathcal{H}_{f}(0,q;a,b) &=&\left( {\frac{f(a^{q},b^{q})}{f(1,1)}}\right)
^{1/q}\text{ if }p=0,q\neq 0,  \label{2.4} \\
\mathcal{H}_{f}(0,0;a,b) &=&a^{\frac{f_{x}(1,1)}{f(1,1)}}b^{\frac{f_{y}(1,1)%
}{f(1,1)}}\text{ if }p=q=0,  \label{2.5}
\end{eqnarray}%
and, $\mathcal{H}_{f}(p,q;a,a)=a$.
\end{definition}

\begin{remark}
Witkowski \cite{Witokwski.12(1).2009.Art.3} proved that if $f(x,y)$ is a
symmetric and $1$-order homogeneous function, then for all $p,q$ $\mathcal{H}%
_{f}(p,q;a,b)$ is a mean of positive numbers $x$ and $y$ if and only if $%
f(x,y)$ is increasing in both variables on $\mathbb{R}_{+}$. In fact, it is
easy to see that the condition "$f(x,y)$ is symmetry" can be removed.
\end{remark}

For simpleness, $\mathcal{H}_{f}(p,q;a,b)$ is also denoted by $\mathcal{H}%
_{f}(p,q)$ or $\mathcal{H}_{f}(a,b)$.

From the published relative literatures, it is complex and difficult to
investigate the monotonicity and log-convexity of $\mathcal{H}_{f}(p,q;a,b)$
for a concrete function $f$ such as $f=L,A,I$. But it becomes to be simple
and easy by using well properties of homogeneous function, which have played
an important role in \cite{Yang.6(4).101.2005}, \cite{Yang.10(3).2007}. We
will show these properties is still useful and efficient in the study of
bivariate log-convexity of $\mathcal{H}_{f}(p,q;a,b)$ with respect to
parameters.

\begin{lemma}[{\protect\cite[Lemma 3, 4]{Yang.10(3).2007}}]
\label{lem2.1}{Let $f:\mathbb{R}_{+}^{2}\backslash \{(x,x):x\in \mathbb{R}%
_{+}\}\rightarrow \mathbb{R}_{+}$}$\ $be a homogenous and three-time
differentiable function. Denote by $T(t):=\ln f(a^{t},b^{t})$ ($t\neq 0)$,
then%
\begin{eqnarray}
{T}^{\prime }(t) &=&\frac{a^{t}f_{x}(a^{t},b^{t})\ln
a+b^{t}f_{y}(a^{t},b^{t})\ln b}{f(a^{t},b^{t})}  \notag \\
{T}^{\prime \prime }(t) &=&-xy\mathcal{I}\ln ^{2}(b/a),\text{where }\mathcal{%
I}=(\ln f)_{xy},  \notag \\
{{T}^{\prime \prime \prime }(t)} &{=}&{-Ct^{-3}\mathcal{J},}\text{ where }{{%
\mathcal{J}=(x-y)(x}\mathcal{I}{)_{x},}C=}\frac{{xy\ln }^{3}{(x/y)}}{{x-y}}{%
>0,}  \label{2.8}
\end{eqnarray}%
where $x=a^{t}$, $y=b^{t}$.
\end{lemma}

\begin{remark}
It follows from (\ref{2.8}) that 
\begin{equation}
\limfunc{sgn}({{T}^{\prime \prime \prime }(t)})=\mathcal{-}\limfunc{sgn}({t})%
\limfunc{sgn}(\mathcal{J)}.  \label{2.9}
\end{equation}
\end{remark}

The following property is also crucial in the proof of Theorem \ref{th2.2}.

\begin{lemma}[{\textbf{\protect\cite[Property 4]{Yang.10(3).2007}}}]
\label{lem2.2}{If }${{T}^{\prime }(t)}$ is continuous on $[q,p]$ or $[p,q]$.
Then ${\ln }\mathcal{H}_{f}(p,q)$ can be expressed in integral form as 
\begin{equation}
{\ln }\mathcal{H}_{f}(p,q)=\left\{ 
\begin{array}{ll}
\dfrac{1}{p-q}{\int_{q}^{p}{T}^{\prime }(t)d}t\text{ } & \bigskip \text{if }%
p\neq q \\ 
{{T}^{\prime }(q)}\text{ } & \text{\bigskip if }p=q%
\end{array}%
\right. ={\int_{0}^{1}{T}^{\prime }(tp+(1-t)q)dt.}  \label{2.10}
\end{equation}
\end{lemma}

We now recall the definition\ and criterion theorem of bivariate convex
(concave) functions.

\begin{definition}
Let $\Omega \subseteq \mathbb{R}^{2}$ be a convex domain and $E\subseteq 
\mathbb{R}$. Then a bivariate function $\Phi :\Omega \rightarrow E$ is said
to be convex (concave) on $\Omega $ if for any pairs $%
(x_{1},y_{1}),(x_{2},y_{2})\in \Omega $ and $\alpha ,\beta >0$ with $\alpha
+\beta =1$ the following inequality 
\begin{equation}
\Phi (\alpha x_{1}+\beta x_{2},\alpha y_{1}+\beta y_{2})\leq (\geq )\alpha
\Phi (x_{1},y_{1})+\beta \Phi (x_{2},y_{2})  \label{2.11}
\end{equation}%
holds.
\end{definition}

\begin{theorem}
\label{th2.1}Let $\Omega \subseteq \mathbb{R}^{2}$ be an open domain and $%
E\subseteq \mathbb{R}$ and the bivariate function $\Phi :\Omega \rightarrow
E $ is two-time differentiable. Then $\Phi (x,y)$ is strictly convex
(concave) over $\Omega $ if and only if both the following inequalities 
\begin{equation}
\Phi _{xx}>(<)0\text{ \ \ and\ }\Delta =\Phi _{xx}\Phi _{yy}-\Phi _{xy}^{2}>0
\label{2.12}
\end{equation}%
hold.
\end{theorem}

Next we are ready to give and prove criterion theorem for the bivariate
log-convexity of two-parameter homogeneous functions with respect to
parameter pair.

\begin{theorem}
\label{th2.2}Let $f:\mathbb{R}_{+}^{2}\backslash \{(x,x),x\in \mathbb{R}%
_{+}\}\rightarrow \mathbb{R}_{+}$ be a homogenous and three-time
differentiable function. If 
\begin{equation*}
\mathcal{J}=(x-y)(x\mathcal{I})_{x}<(>)0\text{, where }\mathcal{I}=\left(
\ln f\right) _{xy},
\end{equation*}%
then the bivariate function $(p,q)\rightarrow \mathcal{H}_{f}(p,q)$ is
strictly log-convex (log-concave) on $\mathbb{R}_{+}^{2}$ and log-concave
(log-convex) on $\mathbb{R}_{-}^{2}$.
\end{theorem}

\begin{proof}
It suffices to verify (\ref{2.12}) holds.

(1) When $(p,q)\in \mathbb{R}_{+}^{2}$. It is obvious to (\ref{2.10}) be
true. Two times partial derivative calculations for (\ref{2.10}) lead to{\ 
\begin{eqnarray*}
\frac{\partial ^{2}\ln \mathcal{H}_{f}(p,q)}{\partial p^{2}} &=&{%
\int_{0}^{1}t^{2}{T}^{\prime \prime \prime }(tp+(1-t)q)dt,} \\
\frac{\partial ^{2}\ln \mathcal{H}_{f}(p,q)}{\partial q^{2}} &=&{%
\int_{0}^{1}(1-t)^{2}{T}^{\prime \prime \prime }(tp+(1-t)q)dt,} \\
\frac{\partial ^{2}\ln \mathcal{H}_{f}(p,q)}{\partial p\partial q} &=&{%
\int_{0}^{1}t(1-t){T}^{\prime \prime \prime }(tp+(1-t)q)dt.}
\end{eqnarray*}%
}

By (\ref{2.9}), for $t\in \lbrack 0,1]$ and $(p,q)\in \mathbb{R}_{+}^{2}$ we
see ${{T}^{\prime \prime \prime }(tp+(1-t)q)>(<)0}$ if $\mathcal{J}=(x-y)(x%
\mathcal{I})_{x}<(>)0$, it follows that $\frac{\partial ^{2}\ln \mathcal{H}%
_{f}(p,q)}{\partial p^{2}}={\int_{0}^{1}t^{2}{T}^{\prime \prime \prime
}(tp+(1-t)q)dt}>(<)0$.

On the other hand, using Cauch-Scharz inequality gives%
\begin{eqnarray*}
\Delta &=&\frac{\partial ^{2}\ln \mathcal{H}_{f}(p,q)}{\partial p^{2}}\frac{%
\partial ^{2}\ln \mathcal{H}_{f}(p,q)}{\partial q^{2}}-\left( \frac{\partial
^{2}\ln \mathcal{H}_{f}(p,q)}{\partial p\partial q}\right) ^{2} \\
&=&{\int_{0}^{1}t^{2}{T}^{\prime \prime \prime }(tp+(1-t)q)dt}\text{ }{%
\int_{0}^{1}(1-t)^{2}{T}^{\prime \prime \prime }(tp+(1-t)q)dt} \\
&&-\left( {\int_{0}^{1}t(1-t){T}^{\prime \prime \prime }(tp+(1-t)q)dt}\text{ 
}\right) ^{2}\geq 0.
\end{eqnarray*}%
With equality if and only if $t=k(1-t)$ for all $t\in \lbrack 0,1]$, where $%
k $ is a constant. But this is obviously impossible, which means the above
inequality is strict.

Applying Theorem \ref{th2.1}, it follows that bivariate function $%
(p,q)\rightarrow \ln \mathcal{H}_{f}(p,q)$ is strictly convex (concave) on $%
(0,\infty )\times (0,\infty )$ if $\mathcal{J}=(x-y)(x\mathcal{I})_{x}<(>)0$%
, which is exactly our result required.

(2) When $(p,q)\in \mathbb{R}_{-}^{2}$. By (\ref{2.9}) we see ${{T}^{\prime
\prime \prime }(tp+(1-t)q)<(>)0}$ for $t\in \lbrack 0,1]$ if $\mathcal{J}%
=(x-y)(x\mathcal{I})_{x}<(>)0$. It follows that\ $\frac{\partial ^{2}\ln 
\mathcal{H}_{f}(p,q)}{\partial p^{2}}<(>)0$. Also, $\Delta \geq 0$ is still
valid and strict. From Theorem \ref{th2.1} it follows that bivariate
function $(p,q)\rightarrow \ln \mathcal{H}_{f}(p,q)$ is strictly concave
(convex) on $\mathbb{R}_{-}^{2}$ if $\mathcal{J}=(x-y)(x\mathcal{I}%
)_{x}<(>)0 $.

This completes the proof.
\end{proof}

\section{\textsc{Main Results}}

Equipped with Theorem \ref{th2.2}, we are in a position to solve the problem
for bivariate log-convexity of the more extended means $\boldsymbol{F}%
(p,q;r,s;a,b)$ in a parameter pair $(p,q)$.

\begin{theorem}
\label{th3.1}For fixed $(r,s)\in \mathbb{R}^{2},(a,b)\in \mathbb{R}_{+}^{2}$%
, the more extended means $\boldsymbol{F}(p,q;r,s;a,b)$ are strictly
bivariate log-concave (log-convex) on $[0,\infty )\times \lbrack 0,\infty )$
and strictly bivariate log-convex (log-concave) on $(-\infty ,0]\times
(-\infty ,0]$ with respect to $(p,q)$ if $r+s>(<)0$.
\end{theorem}

\begin{proof}
It follows from \cite[Proof of Theorem 2.3]{Yang.149286.2008} that $\mathcal{%
J}=(x-y)(x\mathcal{I})_{x}<(>)0$ if $r+s<(>)0$ for $f=S_{r,s}(x,y)$.
Applying Theorem \ref{th2.2} and noting that clearly $\boldsymbol{F}%
(p,q;r,s;a,b)$ are continuous at $(p,q)=(c_{1},0),(0,c_{2}),(0,0)$, $%
c_{1}c_{2}\neq 0$, our desired result follows.
\end{proof}

With $(r,s)=(1,0),(2,1),(1,1),(3/2,1/2)$, the following corollary are
immediate.

\begin{corollary}
\label{Cor.3.1}Stolarsky means $S_{p,q}(a,b)$, Gini means $G_{p,q}(a,b)$,
two-parameter identric (exponential) means $I_{p,q}(a,b)$ and two-parameter
Heronian means $He_{p,q}(a,b)$ are all strictly bivariate log-concave on $%
[0,\infty )\times \lbrack 0,\infty )$ and strictly bivariate log-convex on $%
(-\infty ,0]\times (-\infty ,0]$ with respect to $(p,q)$.
\end{corollary}

\section{\textsc{Applications}}

Many classical and new inequalities for bivariate means can easily follow
from Theorem \ref{th3.1}. For simpleness and convenience of statements in
the sequel, we denote $\boldsymbol{F}(p,q;r,s;a,b)$ by $\boldsymbol{F}(p,q)$%
, and note that for $pr-ps\neq 0$ 
\begin{eqnarray*}
\mathbf{F}(p,2p) &=&\left( \frac{a^{pr}+b^{pr}}{a^{ps}+b^{ps}}\right)
^{1/(pr-ps)}=G_{pr,ps}(a,b), \\
\mathbf{F}(p,0) &=&\left( {\frac{L(a^{pr},b^{pr})}{L(a^{ps},b^{ps})}}\right)
^{\frac{1}{p(r-s)}}=S_{pr,ps}(a,b), \\
\boldsymbol{F}(p,p) &=&\left( {\frac{I(a^{pr},b^{pr})}{I(a^{ps},b^{ps})}}%
\right) ^{\frac{1}{p(r-s)}}=I_{pr,ps}(a,b), \\
\mathbf{F}(p/2,3p/2) &=&\left( {\dfrac{a^{pr}+(\sqrt{ab})^{pr}+b^{pr}}{%
a^{ps}+(\sqrt{ab})^{ps}+b^{ps}}}\right) ^{1/(pr-ps)}=He_{pr,ps}(a,b);
\end{eqnarray*}%
for $pr-ps\neq 0$ they are defined as corresponding limits.

The following are some practical examples of Theorem \ref{th3.1}.

\begin{Application}
\textbf{Some generalizations of classical inequalities for means.}

\begin{itemize}
\item \textbf{Generalized Lin Inequality}. From bivariate log-convexity of
Stolarsky means it follows that $\sqrt{\mathbf{F}(1/3,2/3)\mathbf{F}(1,0)}%
\leq \mathbf{F}(2/3,1/3)$, simplifying yields $\mathbf{F}(1,0)\leq \mathbf{F}%
(1/3,2/3)$, that is, 
\begin{equation}
S_{r,s}(a,b)\leq G_{r/3,s/3}(a,b).  \label{3.1a}
\end{equation}%
Putting $(r,s)=(1,0)$, $(1,1)$ yield $L(a,b)\leq A_{1/3}(a,b)$, $I(a,b)\leq
Z_{1/3}(a,b)$. In which the former is Lin inequality \cite[Theorem 1]%
{Ling.81.1974}, the latter was proved first by Yang \cite[(5.16)]%
{Yang.10(3).2007}.

\item \textbf{Generalized Jia-Cao Inequality. }Similarly, $\mathbf{F}%
^{2/3}(1,0)\mathbf{F}^{1/3}(1/4,3/4)\leq $ $\mathbf{F}(3/4,1/4)$ yields $%
\mathbf{F}(1,0)\leq \mathbf{F}(3/4,1/4)$, that is, 
\begin{equation}
S_{r,s}(a,b)\leq He_{r/2,s/2}(a,b).  \label{3.1b}
\end{equation}%
Putting $(r,s)=(1,0)$ yields $L(a,b)\leq He_{1/2}(a,b)$, which is Jia-Cao
inequality \cite[Theorem 1.1]{Jiagao.4(4).2003}.

\item \textbf{Generalized S\'{a}ndor Inequality }\cite{Sandor.15(1997)}:
From $\sqrt{\mathbf{F}(2,0)\mathbf{F}(0,2)}\leq \mathbf{F}(1,1)\ $it follows
that $\mathbf{F}(1,1)\geq \mathbf{F}(2,0)$, that is, 
\begin{equation}
I_{r,s}(a,b)\geq S_{2r,2s}(a,b).  \label{3.1c}
\end{equation}%
Putting $(r,s)=(1,0)$, $(1,1)$, $(2,1)$ yield $I(a,b)\geq L_{2}(a,b)$, $%
Y(a,b)\geq I_{2}(a,b)$, $Z(a,b)\geq A_{2}(a,b)$, here we have used the
formula $I(a^{2},b^{2})/I(a,b)=Z(a,b)$ (see \cite{Sandor.38.1993}, \cite%
{Yang.6(4).101.2005}). In which the first and the third inequality due to S%
\'{a}ndor \cite{Sandor.40.1990}, \cite{Sandor.15(1997)}, the second one was
presented by Yang \cite[(5.12)]{Yang.10(3).2007}\textbf{.}
\end{itemize}
\end{Application}

\begin{Application}
\textbf{Some new inequalities for two-parameter means.}

\begin{itemize}
\item \textbf{New inequality 1.} The bivariate log-convexity of Stolarsky
means yields $\mathbf{F}^{3/4}(2/3,1/3)\mathbf{F}^{1/4}(1,0)\leq \mathbf{F}%
(3/4,1/4)$, it follows that 
\begin{equation*}
\mathbf{F}(1,0)\leq \mathbf{F}^{4}(3/4,1/4)\mathbf{F}^{-3}(2/3,1/3),
\end{equation*}%
that is, 
\begin{equation}
S_{r,s}(a,b)\leq He_{r/2,s/2}^{4}(a,b)G_{r/3,s/3}^{-3}(a,b).  \label{3.1d}
\end{equation}%
Putting $(r,s)=(1,0)$ yields $L(a,b)\leq He_{1/2}^{4}(a,b)A_{1/3}^{-3}(a,b)$
given in \cite[(4.5)]{Yang.1(1).2010}, which is obviously superior to Lin
and Jia-Cao\textbf{\ }inequality because $L\leq He_{1/2}\leq A_{1/3}$, .

\item \textbf{New inequality 2.} From $\mathbf{F}^{1/5}(1,1)\mathbf{F}%
^{4/5}(3/4,1/4)\leq \mathbf{F}(4/5,2/5)$ it follows that $\mathbf{F}%
(1,1)\leq \mathbf{F}^{5}(4/5,2/5)\mathbf{F}^{-4}(3/4,1/4)$, which yields, 
\begin{equation}
I_{r,s}(a,b)\leq G_{2r/5,2s/5}^{5}(a,b)He_{r/2,s/2}^{-4}(a,b).  \label{3.1e}
\end{equation}%
Putting $(r,s)=(1,0)$, $(2,1)$ yield $I(a,b)\leq
A_{2/5}^{5}(a,b)He_{1/2}^{-4}(a,b)$, $Z(a,b)\leq
G_{4/5,2/5}^{5}(a,b)He_{1,1/2}^{-4}(a,b)$, here we have used the formula $%
I(a^{2},b^{2})/I(a,b)=Z(a,b)$ again.
\end{itemize}
\end{Application}

Lastly, We close this paper by giving a result for bivariate log-convexity
of the two-parameter homogeneous function of difference \cite%
{Yang.6(4).101.2005}\emph{\ }and illustrating its applications.

Substituting $f(x,y)=D(x,y)=|x-y|$ ($x,y>0$ with $x\neq y$) into the
two-parameter homogeneous function $\mathcal{H}_{f}(p,q;a,b)$ leads to 
\begin{equation}
\mathcal{H}_{D}(p,q;a,b):=\left\{ 
\begin{array}{ll}
\Big|{\dfrac{a^{p}-b^{p}}{a^{q}-b^{q}}}\Big|^{1/(p-q)} & \text{if }p\neq q{,}
\\ 
e^{1/p}I^{1/p}(a^{p},b^{p}) & \text{if }p=q,%
\end{array}%
\right.  \label{3.4}
\end{equation}%
which is called two-parameter homogeneous function of difference.

For $\mathcal{H}_{D}(p,q;a,b)$, we have shown $\mathcal{J}=(x-y)(x\mathcal{I}%
)_{x}<0$, where $\mathcal{I=(}\ln \mathcal{D)}_{xy}$, in \cite[pp. 510]%
{Yang.10(3).2007} and by Theorem \ref{th2.2} we obtain

\begin{theorem}
\label{th3.2}$\mathcal{H}_{D}(p,q;a,b)$ is strictly bivariate log-convex on $%
(0,\infty )\times (0,\infty )$ and bivariate log-convex on $(-\infty
,0)\times (-\infty ,0)$ with respect to $(p,q)$.
\end{theorem}

By Corollary \ref{Cor.3.1} and Theorem \ref{th3.2}, we have the following.

\begin{corollary}
\label{Cor.3.2}For any $(a,b),$ $(p_{1},q_{1}),(p_{2},q_{2})\in \mathbb{R}%
_{+}^{2}$ and $\alpha ,\beta >0$ with $\alpha +\beta =1$, the following
double inequalities%
\begin{equation}
1\leq \frac{S_{\alpha p_{1}+\beta p_{2},\alpha q_{1}+\beta q_{2}}(a,b)}{%
S_{p_{1},q_{1}}^{\alpha }(a,b)S_{p_{2},q_{2}}^{\beta }(a,b)}\leq e^{\tfrac{%
\alpha }{L(p_{1},q_{1})}+\tfrac{\beta }{L(p_{2},q_{2})}-\tfrac{1}{L(\alpha
p_{1}+\beta p_{2},\alpha q_{1}+\beta q_{2})}}  \label{3.5}
\end{equation}%
hold.
\end{corollary}

\begin{proof}
Corollary \ref{Cor.3.1} and Theorem \ref{th3.2} yield 
\begin{eqnarray}
S_{p_{1},q_{1}}^{\alpha }(a,b)S_{p_{2},q_{2}}^{\beta }(a,b) &\leq &S_{\alpha
p_{1}+\beta p_{2},\alpha q_{1}+\beta q_{2}}(a,b),  \label{3.6} \\
\mathcal{H}_{D}^{\alpha }(p_{1},q_{1};a,b)\mathcal{H}_{D}^{\beta
}(p_{2},q_{2};a,b) &\geq &\mathcal{H}_{D}(\alpha p_{1}+\beta p_{2},\alpha
q_{1}+\beta q_{2};a,b),  \label{3.7}
\end{eqnarray}%
respectively. (\ref{3.6}) implies that the first inequality of (\ref{3.5}).
To obtain the second one, note $\mathcal{H}%
_{D}(p,q;a,b)=(p/q)^{1/(p-q)}S_{p,q}(a,b)=e^{1/L(p,q)}S_{p,q}(a,b)$ ($p,q>0$%
) and substitute into (\ref{3.7}), thus (\ref{3.7}) can be written as 
\begin{equation*}
e^{\frac{\alpha }{L(p_{1},q_{1})}+\frac{\beta }{L(p_{2},q_{2})}%
}S_{p_{1},q_{1}}^{\alpha }(a,b)S_{p_{2},q_{2}}^{\beta }(a,b)\geq e^{\frac{1}{%
L(\alpha p_{1}+\beta p_{2},\alpha q_{1}+\beta q_{2})}}S_{\alpha p_{1}+\beta
p_{2},\alpha q_{1}+\beta q_{2}}(a,b),
\end{equation*}%
which is equivalent to the second one of (\ref{3.5}).

The proof is complete.
\end{proof}

For Gini means, we also have similar result.

\begin{corollary}
\label{Cor.3.3}For any $(a,b),$ $(p_{1},q_{1}),(p_{2},q_{2})\in \mathbb{R}%
_{+}^{2}$ and $\alpha ,\beta >0$ with $\alpha +\beta =1$, the following
double inequality%
\begin{equation}
1\leq \frac{G_{\alpha p_{1}+\beta p_{2},\alpha q_{1}+\beta q_{2}}(a,b)}{%
G_{p_{1},q_{1}}^{\alpha }(a,b)G_{p_{2},q_{2}}^{\beta }(a,b)}\leq e^{\tfrac{%
\alpha }{L(p_{1},q_{1})}+\tfrac{\beta }{L(p_{2},q_{2})}-\tfrac{1}{L(\alpha
p_{1}+\beta p_{2},\alpha q_{1}+\beta q_{2})}}  \label{3.8}
\end{equation}%
holds.
\end{corollary}

\begin{proof}
The first inequality follows from bivariate log-convexity of Gini means,
that is, Corollary \ref{Cor.3.1}. It remains to prove the second one. Form
Theorem \ref{th3.2} we have%
\begin{equation*}
\mathcal{H}_{D}^{\alpha }(2p_{1},2q_{1};a,b)\mathcal{H}_{D}^{\beta
}(2p_{2},2q_{2};a,b)\geq \mathcal{H}_{D}(2\alpha p_{1}+2\beta p_{2},2\alpha
q_{1}+2\beta q_{2};a,b).
\end{equation*}%
It is easy to verify that$\mathcal{\ H}_{D}(p,q;a,b)G_{p,q}(a,b)=\mathcal{H}%
_{D}^{2}(2p,2q;a,b)$, which is applied to the above inequality and a simple
equivalent transformation lead to 
\begin{equation*}
\frac{G_{\alpha p_{1}+\beta p_{2},\alpha q_{1}+\beta q_{2}}(a,b)}{%
G_{p_{1},q_{1}}^{\alpha }(a,b)G_{p_{2},q_{2}}^{\beta }(a,b)}\leq \frac{%
\mathcal{H}_{D}^{\alpha }(p_{1},q_{1};a,b)\mathcal{H}_{D}^{\beta
}(p_{2},q_{2};a,b)}{\mathcal{H}_{D}(\alpha p_{1}+\beta p_{2},\alpha
q_{1}+\beta q_{2};a,b)}.
\end{equation*}%
Using relation $\mathcal{H}_{D}(p,q;a,b)=e^{1/L(p,q)}S_{p,q}(a,b)$ ($p,q>0$)
mentioned previous again and the first inequality of (\ref{3.5}), the right
side of the above inequality is lees than or equal to 
\begin{equation*}
e^{\tfrac{\alpha }{L(p_{1},q_{1})}+\tfrac{\beta }{L(p_{2},q_{2})}-\tfrac{1}{%
L(\alpha p_{1}+\beta p_{2},\alpha q_{1}+\beta q_{2})}.}
\end{equation*}

The proof ends.
\end{proof}

From Corollary \ref{Cor.3.2}, \ref{Cor.3.3}, we can easily estimate the
upper and lower bounds of the ratio of means. Some of which are contained in
the following.

\begin{Application}
\textbf{Some classical double estimation inequalities.}

\begin{itemize}
\item \textbf{Stolarsky-Yang Inequality }\cite[(1.6)]{Stolarsky.87.1980}%
\textbf{,} \cite[Remark 9]{Yang.10(3).2007}.

Putting $(p_{1},q_{1})=(4/3,2/3),(p_{2},q_{2})=(2/3,4/3),(\alpha ,\beta
)=(1/2,1/2)$ in (\ref{3.5}) and simplifying yield 
\begin{equation}
1\leq I/A_{2/3}\leq \sqrt{8}e^{-1}\approx 1.0405.  \label{3.9}
\end{equation}

\item \textbf{S\'{a}ndor-Yang Inequality }\cite{Sandor.38.1993}, \cite[%
Remark 9]{Yang.10(3).2007}\textbf{.}

Putting $(p_{1},q_{1})=(1/2,3/2),(p_{2},q_{2})=(3/2,1/2),(\alpha ,\beta
)=(1/6,5/6)$ in (\ref{3.5}) and simplifying yield 
\begin{equation}
1\leq A_{2/3}/He\leq 3/\sqrt{8}\approx 1.0607.  \label{3.10}
\end{equation}
\end{itemize}
\end{Application}

\begin{Application}
\textbf{Some new double estimation inequalities}.

\begin{itemize}
\item \textbf{New double estimation inequalities 1.} Putting $%
(p_{1},q_{1})=(1,1),(p_{2},q_{2})=(3/2,1/2),(\alpha ,\beta )=(1/3,2/3)$ in (%
\ref{3.5}) and simplifying and some simple equivalent transformations yield 
\begin{equation}
0.9249\approx 16\sqrt{2}e^{-1}/9\leq I/A_{2/3}^{3}He^{-2}\leq 1.
\label{3.11}
\end{equation}

\item \textbf{New double estimation inequalities 2.} Putting $%
(p_{1},q_{1})=(6/5,6/5),(p_{2},q_{2})=(4/5,4/5),(\alpha ,\beta )=(1/2,1/2)$
in (\ref{3.5}), (\ref{3.8}) and simplifying yield 
\begin{eqnarray}
1 &\leq &I/\sqrt{I_{6/5}I_{4/5}}\leq e^{1/24}\approx 1.0425,  \label{3.12} \\
1 &\leq &Z/\sqrt{Z_{6/5}Z_{4/5}}\leq e^{1/24}\approx 1.0425.  \label{3.13}
\end{eqnarray}

\item \textbf{New double estimation inequalities 3. }Putting $%
(p_{1},q_{1})=(1/2,3/2),(p_{2},q_{2})=(3/2,1/2),(\alpha ,\beta )=(1/2,1/2)$
in (\ref{3.8}) and simplifying yield 
\begin{equation}
1\leq Z/(2A-G)\leq 3e^{-1}\approx 1.1036.  \label{3.14}
\end{equation}
\end{itemize}
\end{Application}

Additionally, many other unknown and interesting inequalities for means can
be derived from theorems and corollaries in this section and no longer list
here.

\end{document}